\documentclass[a4paper,12pt]{article} 
\usepackage{graphicx}
\usepackage[english]{babel} 
\usepackage{amsmath,amssymb,amsthm,amsfonts} 
\usepackage{latexsym}
\usepackage[all]{xy}
\usepackage{url}
\usepackage{mathrsfs}

\newcommand{\Q}{\mathbb Q}

\newcommand{\Z}{\mathbb Z}
\newcommand{\C}{\mathbb C}

\newcommand{\p}{\mathfrak p}

\newcommand{\gal}{\mathrm{Gal}}
\renewcommand{\epsilon}{\varepsilon}

\newcommand{\rt}{\mathrm{R}_t}
\newcommand{\cl}{\mathrm{Cl}}
\newcommand{\perm}{\mathrm{Perm}}
\newcommand{\aut}{\mathrm{Aut}\ }
\newcommand{\ind}{\mathrm{ind}}
\newcommand{\res}{\mathrm{res}}
\renewcommand{\hom}{\mathrm{Hom}}
\newcommand{\map}{\mathrm{Map}}
\newcommand{\oG}{\overline{G}}
\newcommand{\oL}{\overline{\Lambda}}

\newcommand{\oo}{\mathcal O}
\newcommand{\W}{\mathcal W}
\newcommand{\ab}{\mathrm{ab}}
\newcommand{\gl}{\mathrm{GL}}
\newcommand{\h}{\mathscr{H}}
\newcommand{\Det}{\mathrm{Det}}

\newtheorem{coroll}{Corollary}[equation]
\newtheorem{prop}[equation]{Proposition}
\newtheorem*{prop7}{Proposition 7$\,'$}

\theoremstyle{remark}
\newtheorem*{remark}{Remark}

\theoremstyle{definition}
\newtheorem*{defn}{Definition}

\title{From Galois module classes to Steinitz classes}
\author{Leon R. McCulloh}
\date{}
\begin{document}
\maketitle

\emph{This article is a transcription of some handwritten notes of an informal report given by Leon R. McCulloh in Oberwolfach in February 2002.}

\begin{abstract}
Let $G$ be a finite group of order $n$ and exponent $e$. Corresponding to a tame Galois extension $L/K$ of number fields with an isomorphism $\gal(L/K)\cong G$, one has associated the Galois module class $\cl_{\oo_KG}(\oo_L)$ (resp., the Steinitz class $\cl_{\oo_K}(\oo_L)$) in the class group $\cl(\oo_KG)$ (resp., $\cl(\oo_K )$). For fixed number field $K$, the set of Galois module (resp., Steinitz) classes realized in $\cl(\oo_KG)$ (resp., $\cl(\oo_K)$) as $L/K$ ranges over all tame Galois extensions with isomorphism $\gal(L/K)\cong G$ is denoted by $R(\oo_KG)$ (resp., $R_t(\oo_K,G)$). We show 
\textbf{Theorem} a) 
\[R_t(\oo_K,G)\subseteq\prod_{m|e}\prod_{\substack{s\in G\\|s|=m}}N(K(\bar s)/K)^{\frac{n}{m}\frac{m-1}{2}}\]
unless the Sylow $2$-subgroups of $G$ are non­trivially cyclic, in which case\\
b) 
\[R_t(\oo_K,G)^2\subseteq\prod_{m|e}\prod_{\substack{s\in G\\|s|=m}}N(K(\bar s)/K)^{\frac{n}{m}(m-1)}\]
where $N(K(\bar s)/K)=N_{K(\bar s)/K}(\cl(\oo_{K(\bar s)}))\subseteq \cl(\oo_K)$. 
Further explaining, $\bar G$ is the set of conjugacy classes of $G$ endowed with a ``cyclotomic'' action of $\Omega\ (=\gal(K^c/K))$ via $\kappa:\Omega\to\gal(K(\zeta_e)/K)\hookrightarrow(\Z/e\Z)^\times$ - explicitly, for $s\in G$ and corresponding $\bar s\in\bar G$ and $\omega\in\Omega$, $s^\omega=s^{\kappa^{-1}(\omega)}$ and $\bar s^\omega=\overline{s^\omega}$, where $\zeta_e^\omega=\zeta_e^{\kappa(\omega)}$. Finally, $K(\bar s)=(K^c)^{\Omega_{\bar s}}$, where $\Omega_{\bar s}$ is the $\Omega$-stabilizer of $\bar s$.

The result is obtained from a known inclusion of $R(\oo_KG)$ via $\res^G_1:\cl(\oo_KG)\to\cl(\oo_K)$. 
\end{abstract}

Let $G$ be a finite group, $K\subseteq \C$ an algebraic number field and $\Omega=\Omega_K=\gal(K^c/K)$, where $K^c\subseteq \C$ is the algebraic closure of $K$.
\begin{defn}
If $w:\Omega\to\aut G$ is a homomorphism, we denote by $G_w$ the $\Omega$-group $G$ with $\Omega$-action given by $s^\omega=s^{w(\omega)}$ for $s\in G_w$, $\omega\in\Omega$.
\end{defn}
Let $R_G$ be the virtual character ring of $G$. Then $\Omega$ acts on $R_G$ naturally by
\[\chi^\omega(s)=\chi(s)^\omega\qquad \text{for }s\in G,\omega\in\Omega,\]
where $\chi$ is an absolutely irreducible character of $G$. Let $\Q R_G=\Q\otimes_\Z R_G$.
\begin{defn}
We define a $\Q$-pairing $\langle\ ,\ \rangle:\Q R_G\times\Q G\to \Q$ as follows.
\begin{enumerate}
\item[a)] If $\chi$ is a character of degree $1$ and $s\in G$, then $\langle\chi,s\rangle$ is the rational number defined by $\chi(s)=e^{2\pi i\langle\chi,s\rangle}$, where $0\leq \langle\chi,s\rangle<1$.

(Note: If $G$ is abelian, then a) defines $\langle\ ,\ \rangle:\Q R_G\times\Q G\to \Q$ by $\Q$-bilinear extension.)
\item[b)] If $\chi$ is any character of $G$, then $\langle\chi,s\rangle$ is defined by
\[\langle\chi,s\rangle=\langle\res^G_{\langle s\rangle}\chi,s\rangle.\]
(Note: $\Q$-bilinearity extends this to the required pairing.)
\end{enumerate}
\end{defn}

\begin{defn}
We define the Stickelberger map $\Theta_G:\Q R_G\to \Q G$ by
\[\Theta_G(\chi)=\sum_{s\in G}\langle\chi,s\rangle s\quad\text{for }\chi\in\Q R_G.\]
\end{defn}

\begin{defn}
The Stickelberger module $S_G\subseteq \Z G$ is given by
\[S_G=\Theta_G(R_G)\cap \Z G.\]
\end{defn}

\begin{defn}
If $\chi$ is a character of $G$, then $\det \chi$ is a character of $G$ of degree $1$ (or a character of $G^\ab$), where
\[(\det\chi)(s)=\det(T_\chi(s)),\]
where $T_\chi:G\to\gl(\C)$ is any matrix representation affording $\chi$. This extends to a $\Z$-homomorphism
\[\det:R_G\to\widehat{G^\ab}\quad \text{(the character group of $G^\ab$)}.\]
We let $A_G=\ker(\det)$, so we have the exact sequence
\[0\to A_G\to R_G\stackrel{\det}\to \widehat{G^\ab}\to 1.\]
\end{defn}

\begin{prop}\label{thetazg}
If $\chi\in R_G$, then $\Theta_G(\chi)\in\Z G$ if and only if $\chi\in A_G$. In particular $\Theta_G(A_G)=S_G$.
\end{prop}
\begin{proof}
Let $\chi\in R_G$. Then $\Theta_G(\chi)=\sum_{s\in G}\langle\chi,s\rangle s\in \Z G$ if and only if $\langle\chi,s\rangle=\langle\res^G_{\langle s\rangle}\chi,s\rangle \in \Z$ for all $s\in G$. For all $s\in G$ the representation $\res^G_{\langle s\rangle}\chi$ has a decomposition in representations of dimension $1$, which correspond to characters $\psi_{s,1},\dots,\psi_{s,n}$, where $n$ is a positive integer. Hence
\[\begin{split}(\det\chi)(s)&=(\det \res^G_{\langle s\rangle}\chi)(s)=\prod_{j=1}^n\psi_{s,j}(s)=\prod_{j=1}^n e^{2\pi i\langle\psi_{s,j},s\rangle}\\&=e^{2\pi i\sum_{j=1}^n\langle\psi_{s,j},s\rangle}=e^{2\pi i\langle\res ^G_{\langle s\rangle}\chi,s\rangle}.\end{split}\]
Therefore $\langle\res^G_{\langle s\rangle}\chi,s\rangle \in \Z$ for all $s\in G$ if and only if $\chi\in\ker(\det)=A_G$.
\end{proof}

\begin{prop}\label{thetareg}
Let $\rho_G$ be the character of the regular representation of $G$. Let $s\in G$, where $s$ has order $m$ ($|s|=m$) and $\# G=n$. Then
$$\langle\rho_G,s\rangle=\frac{n}{m}\frac{m-1}{2}.$$
\end{prop}
\begin{proof}
We compute
$$\langle\rho_G,s\rangle=\langle\res_{\langle s\rangle}^G\rho_G,s\rangle=\langle[G:\langle s\rangle]\rho_{\langle s\rangle},s\rangle=\frac{n}{m}\langle\rho_{\langle s\rangle},s\rangle.$$
Now, the characters of $\langle s\rangle$ are $\{\phi^j : j=0,\ldots,m-1\}$ where $\phi^j(s)=e^{2\pi i\frac{j}{m}}$, and $\langle\phi^j,s\rangle=\frac{j}{m}$; so
$$
\langle\rho_{\langle s\rangle},s\rangle=\sum_{j=0}^{m-1}\langle\phi^j,s\rangle=\sum_{j=0}^{m-1}\frac{j}{m}=\frac{1}{m}\left(\frac{(m-1)m}{2}\right)=\frac{m-1}{2}.
$$
Hence $\langle\rho_G,s\rangle=\frac{n}{m}\frac{m-1}{2}$.
\end{proof}

\begin{coroll}\label{coroll2.1}
Let $e$ be the exponent of $G$. 
$$\Theta_G(\rho_G)=\frac{n}{e}\sum_{m\mid e}\frac{e}{m}\frac{m-1}{2}\sum_{\substack{s\in G\\ |s|=m}}s=\sum_{m\mid e}\frac{n}{m}\frac{m-1}{2}\sum_{\substack{s\in G\\ |s|=m}}s.$$
\end{coroll}
\begin{proof}
Clear.
\end{proof}

\begin{prop}
$\rho_G\in A_G\ (=\ker(\det))$ unless $n=\# G$ is even and the Sylow $2$-subgroups of $G$ are cyclic (and conversely).
\end{prop}
\begin{proof}
The regular representation of $G$ is given by permutation matrices corresponding to the left regular permutation representation of $G$ on $G$. Their determinants are $+1$ or $-1$ according as the corresponding permutations are even or odd. So
$$(\det \rho_G)(s)=\left\{\begin{array}{cc}+1&\textrm{if multiplication by $s$ is an even permutation of $G$}\\-1&\textrm{if multiplication by $s$ is an odd permutation of $G$.}\end{array}\right.$$ 
If $s$ has order $m$, then the permutation of $G$ given by multiplication by $s$ decomposes into $\frac{n}{m}=[G:\langle s\rangle]$ cycles of length $m=|s|$. So for $s$ to give an odd permutation, $m$ must be even and $\frac{n}{m}$ odd. This happens only if $n$ is even and $\langle s\rangle$ contains a Sylow $2$-subgroup of $G$, that is the Sylow $2$-subgroups of $G$ are cyclic. (Conversely if $\langle s\rangle$ is the Sylow $2$-subgroup of $G$, $s\ne 1$, then $s$ gives an odd permutation of $G$). 
\end{proof}

\noindent
\underline{Note}: In any case $2\rho_G\in A_G$.

\begin{coroll}
$\Theta_G(\rho_G)\in \Z[G]$ unless $G$ has even order and the $2$-Sylow subgroups are cyclic (and conversely).
\end{coroll} 
\begin{proof}
Clear from Proposition \ref{thetazg}.
\end{proof}
\begin{proof}[Alternate proof]
By Corollary \ref{coroll2.1}, $\Theta_G(\rho_G)\in \Z[G]$ unless $G$ has an element $s$ of order $m$, with $\frac{n}{m}(m-1)$ odd. That is $m=|s|$ is even and $\frac{n}{m}=[G:\langle s\rangle]$ is odd. As before, this happens if and only if $n=\# G$ is even and the Sylow $2$-subgroups of $G$ are cylic. 
\end{proof}

\noindent
\underline{Note}: As before, $\Theta_G(2\rho_G)\in\Z[G]$ in all cases.

\begin{defn}
Let $e$ be the exponent of $G$ and $\mu_e$ the group of $e^{\textrm{th}}$ roots of $1$. Let $\kappa:\Omega_K\to (\Z/e\Z)^\times$ be the homomorphism defined via the cyclotomic character
$$\Omega_K\twoheadrightarrow \gal(K(\mu_e)/K)\hookrightarrow (\Z/e\Z)^\times.$$
That is, if $\zeta\in\mu_e$, $\zeta^\omega=\zeta^{\kappa(\omega)}$.
\end{defn}

For each $r\in\Z$ we can define an ``action'' of $\Omega$ on $G$ by permutations by
$$s^\omega=s^{\kappa^r(\omega)}\quad\textrm{for each $s\in G$, $\omega\in \Omega$.}$$
Then (by abuse of notation) we denote by $G_{\kappa^r}$ the $\Omega$-set $G$ with the action given by $\kappa^r:\Omega\to (\Z/e\Z)^\times \to \perm(G)$. Then $\Z G_{\kappa^r}$ is an $\Omega$-module.

\begin{prop}
If $\chi\in R_G$ and $s\in G$, then 
$$\langle\chi^{\omega},s\rangle=\langle\chi, s^{\kappa(\omega)}\rangle.$$
It follows that $\Theta_G:\Q R_G\to \Q G_{\kappa^{-1}}$ is an $\Omega$-homomorphism. That is, $\Theta_G(\chi^\omega)=\Theta_G(\chi)^\omega$ ($\omega\in \Omega$ acting on $G$ via $\kappa^{-1}$). (Note: $\kappa^{-1}(\omega)=\kappa(\omega)^{-1}$ in $(\Z/e\Z)^\times$).
\end{prop}
\begin{proof}
If $\chi$ is of degree $1$, then $\chi(s)\in \mu_e$ (if $s\in G$). So $\chi^\omega(s)=\left(\chi(s)\right)^\omega=\chi(s)^{\kappa(\omega)}=\chi(s^{\kappa(\omega)})$, so 
$$e^{2\pi i\langle\chi^{\omega},s\rangle}=e^{2\pi i\langle\chi,s^{\kappa(\omega)}\rangle}\quad \textrm{i.e. $\langle\chi^\omega,s\rangle=\langle\chi,s^{\kappa(\omega)}\rangle$}.$$ 
Hence by $\Q$-bilinearity, $\langle\chi^\omega,s\rangle=\langle\chi,s^{\kappa(\omega)}\rangle$ for all $\chi\in \Q R_G$, $s\in G$. Thus if $\chi\in \Q R_G$,
$$\Theta_G(\chi^\omega)=\sum_{s\in G}\langle\chi^\omega, s \rangle s=\sum_{s\in G}\langle\chi, s^{\kappa(\omega)} \rangle s = \sum_{s\in G}\langle\chi, s \rangle s^{\kappa^{-1}(\omega)}= \Theta_G(\chi)^{\omega}$$
if $ \Theta_G(\chi)$ is regarded as an element of $\Q G_{\kappa^{-1}}$.
\end{proof}

Note that $\langle\chi,s\rangle$ depends only on the conjugacy class of $s\in G$ so $\Theta_G(R_G)\subseteq c(\Q G)$ (the center of $\Q G$, with basis the conjugacy class sums of $G$). Let $\oG$ be the set of $G$-conjugacy classes of $G$. The action of $\Omega$ by $\kappa^{-1}$ preserves conjugacy classes of $G$, so it gives actions on $c(\Z G)$ and $\Z \oG$, the free $\Z$-module on $\oG$. We denote these $\Omega_K$-modules by $c(\Z G_{\kappa^{-1}})$ and $\Z \oG_{\kappa^{-1}}$, respectively.
 
\begin{defn}
Let $\overline{s}$ be the conjugacy class of $s\in G$. Let 
$$\iota:\Z \oG\to c(\Z G)$$
be the map defined by $\iota(\overline{s})=\sum_{t\in \overline{s}}t$ for all $\overline{s}\in \oG$. Clearly $\iota$ commutes with the action of $\Omega$ by $\kappa^{-1}$ so it is an $\Omega$-module isomorphism
$$\Z \oG\stackrel{\sim}{\longrightarrow} c(\Z G).$$
Also, define $\langle\,,\,\rangle:\Q R_G\times \Q \oG\to \Q$ by $\langle\chi,\overline{s}\rangle=\langle\chi,s\rangle$ for $s\in G$, $\chi\in R_G$ and 
$$\Theta_{\oG}:\Q R_G\to \Q\oG$$
by $\Theta_{\oG}(\chi)=\sum_{\overline{s}\in \oG}\langle\chi,\overline{s}\rangle\overline{s}$.
\end{defn}

Clearly $\iota\Theta_{\oG}=\Theta_G$, and if $\chi\in R_G$, 
$$\Theta_{\oG}(\chi)\in \Z\oG\Longleftrightarrow \chi\in A_G.$$

\begin{defn}
Let $\oL$ be the maximal $\oo_K$-order of the commutative $K$-algebra
$$K\oL:=\map_{\Omega}(\oG_{\kappa^{-1}},K^c).$$
Then $\oL=\map_{\Omega}(\oG_{\kappa^{-1}},\oo^c)$ and we have the ad\`ele ring $\mathbb{A}(K\oL)$ and id\`ele group $\mathbb{J}(K\oL)$ (where $\mathbb{A}(K)$ is the ad\`ele ring of $K$) given as
\begin{eqnarray*}
\mathbb{A}(K\oL)&=&\map_{\Omega}(\oG_{\kappa^{-1}},\mathbb{A}(K^c)) \quad(\mathbb{A}(K^c)=K^c\otimes_K\mathbb{A}(K))\\
\mathbb{J}(K\oL)&=&\hom_{\Omega}(\Z\oG_{\kappa^{-1}},
\mathbb{J}(K^c)) \quad(\mathbb{J}(K^c)=\mathbb{A}(K^c)^\times).
\end{eqnarray*}
Also
\begin{eqnarray*}
(K\oL)^\times&=&\hom_\Omega(\Z \oG_{\kappa^{-1}},(K^c)^\times)\\
\oL^\times&=&\hom_\Omega(\Z \oG_{\kappa^{-1}},(\oo^c)^\times).
\end{eqnarray*}
The integral ad\`eles of $\oL$ are
$$\mathbb{A}(\oL)=\map_\Omega(\oG_{\kappa^{-1}},\mathbb{A}(\oo^c))$$
where $\mathbb{A}(\oo^c)=\oo^c\otimes_{\oo_K}\mathbb{A}(\oo_K)$, $\mathbb{A}(\oo_K)=\prod_{\textrm{$v$ of $K$}}\oo_{K_v}$ and the unit id\`eles of $\oL$ are 
$$\mathbb{U}(\oL)=\mathbb{A}(\oL)^\times=\hom_\Omega(\Z\oG_{\kappa^{-1}}, \mathbb{U}(\oo^c)) \quad (\mathbb{U}(\oo^c)=\mathbb{A}(\oo^c)^\times).$$
\end{defn}
Now the map $\Theta_{\oG}:A_G\to \Z \oG_{\kappa^{-1}}$ induces a homomorphism
\begin{equation}\Theta_{\oG}^t:\hom_\Omega(\Z \oG_{\kappa^{-1}},\mathbb J(K^c))\to\hom_\Omega(A_G,\mathbb J(K^c)).\end{equation}
\begin{defn}
Let $\h(\mathbb{A}(K)G):=\hom_\Omega(A_G,\mathbb J(K^c))$ (and $\h(KG):=\hom_\Omega(A_G,{K^c}^\times)$).
\end{defn}
Then the above map is
\begin{equation}\tag{5'}\Theta_{\oG}^t:\mathbb J(K\oL)\to\h(\mathbb{A}(K)G).\end{equation}
Now recall the map
\[\Det:K^cG^\times\to\hom(R_G,{K^c}^\times)\]
given by $\Det(\alpha):\chi\mapsto\det(T_\chi(\alpha))$, where $T_\chi:G\to\gl_f(K^c)$ affords $\chi$. ($T_\chi$ extends by linearity to a ring homomorphism $T_\chi:K^cG\to M_f(K^c)$ and then to $T_\chi:(K^cG)^\times\to\gl_f(K^c)$.)

Notice that $\Det$ is an $\Omega$-homomorphism.
\begin{prop}\label{prop6}
There is a commutative $\Omega$-diagram (that is, all maps preserve the action of $\Omega$) with exact rows:
\[\xymatrix{
1\ar[r] & G\ar[r]\ar[d]^\Det&K^cG^\times\ar[d]^\Det\ar[r]&(K^cG)^\times/G\ar[r]\ar[d]^\Det&1
\\ 1\ar[r]& G^{\ab}\ar[r]&\hom(R_G,{K^c}^\times)\ar[r]^{\mathrm{rag}}&\hom(A_G,{K^c}^\times)\ar[r]&1.}\]
(Note that $(K^cG)^\times/G$ is not a group, but a pointed set.)
\end{prop}
\begin{proof}
We identify here $G^{\ab}$ with $\hom(\widehat{G^\ab},{K^c}^\times)$. Since ${K^c}^\times$ is injective, the bottom row is exact, arising from the exact sequence $0\to A_G\to R_G\stackrel{\det}\to \widehat{G^\ab}\to 1$. If $s\in G$, $\Det(s)(\chi)=\det(\chi)(s)$, so $\Det$ restricted to $G$ has image $\hom(\widehat{G^\ab},{K^c}^\times)$. Hence $\Det$ induces a map
\[\Det:(K^cG)^\times\to\hom(A_G,{K^c}^\times)\]
making the diagram commute.
\end{proof}
\begin{prop}
We have the commutative diagram
\[\xymatrix{
1\ar[r] & G\ar[r]\ar[d]^\Det&KG^\times\ar[d]^\Det\ar[r]&((K^cG)^\times/G)^\Omega\ar@{->>}[r]\ar[d]^\Det&\hom(\Omega,G)/G\ar[d]
\\ 1\ar[r]& G^{\ab}\ar[r]&\hom_\Omega(R_G,{K^c}^\times)\ar[r]^{\mathrm{rag}}&\hom_\Omega(A_G,{K^c}^\times)\ar[r]&\hom(\Omega,G^\ab)\ar[r]&1}\]
with exact rows (where $G$ acts on $\hom(\Omega,G)$ on the right by inner automorphisms: $h\cdot s=s^{-1}hs, s\in G, h\in\hom(\Omega,G))$.
\end{prop}
\begin{proof}
This follows from Proposition \ref{prop6} by applying $\Omega$-cohomology (in the non-abelian sense for the top row). More precisely, exactness means that if $((K^cG)^\times/G)^\Omega=\h(KG)/G$, then $(KG)^\times\setminus\h(KG)/G$ is in one-one correspondence with $\hom(\Omega,G)/G$.
\end{proof}

Now, similarly, if $v$ is a prime of $K$, and $\Omega_v=\gal(K_v^c/K_v)$, we have
\begin{prop7}
The following diagram with exact rows commutes:
\[\xymatrix{
1\ar[r] & G\ar[r]\ar[d]^\Det&K_vG^\times\ar[d]^\Det\ar[r]&((K_v^cG)^\times/G)^{\Omega_v}\ar@{->>}[r]\ar[d]^\Det&\hom(\Omega_v,G)/G\ar[d]
\\ 1\ar[r]& G^{\ab}\ar[r]&\hom_{\Omega_v}(R_G,{K_v^c}^\times)\ar[r]^{\mathrm{rag}}&\hom_{\Omega_v}(A_G,{K_v^c}^\times)\ar[r]&\hom(\Omega_v,G^\ab)\ar[r]&1}\]
\end{prop7}
Now if $\oo_v$ (resp. $\oo_v^c$) is the ring of integers of $K_v$ (resp. $K_v^c$) then we likewise have the commutative diagram
\begin{equation}\label{diag8}
\xymatrix{
(\oo_vG)^\times\ar[d]^\Det\ar[r]&((\oo_v^cG)^\times/G)^{\Omega_v}\ar[d]^\Det
\\ \hom_{\Omega_v}(R_G,(\oo_v^c)^\times)\ar[r]^{\mathrm{rag}}&\hom_{\Omega_v}(A_G,(\oo_v^c)^\times)}
\end{equation}
\begin{defn}
Let $\h(\mathbb{A}(\oo)G)=\prod_v\Det(((\oo_v^cG)^\times/G)^{\Omega_v})$.
\end{defn}
\begin{prop}\label{prop9}
\[\h(\mathbb{A}(\oo)G)\subseteq\hom_\Omega(A_G,\mathbb{U}(\oo^c))\subseteq\h(\mathbb{A}(K)G).\]
\end{prop}
\begin{proof}
The second inclusion is obvious since $\h(KG)$ was defined as $\hom_\Omega(A_G,\mathbb J(K^c))$. For the first inclusion notice that (see \cite[Remark 6.22]{McCulloh_Crelle})
\[\hom_{\Omega_v}(A_G,(\oo_v^c)^\times)=\hom_\Omega(A_G,(\oo^c)_v^\times),\]
so
\[\begin{split}\prod_v\Det(((\oo_v^cG)^\times/G)^{\Omega_v})&\subseteq \prod_v\hom_{\Omega_v}(A_G,(\oo_v^c)^\times)=\prod_v\hom_\Omega(A_G,(\oo^c)_v^\times)\\
&=\hom_\Omega(A_G,\prod_v(\oo^c)_v^\times)=\hom_\Omega(A_G,\mathbb{U}(\oo^c)).\end{split}\]
\end{proof}

\begin{defn}
Let
\[\mathrm{rag}:\hom_\Omega(R_G,\mathbb{J}(K^c))\to\hom_\Omega(A_G,\mathbb J(K^c))=\h(\mathbb{A}(K)G)\]
be the map "restriction to $A_G$".
\end{defn}

\begin{prop}
The map $\mathrm{rag}$ induces a homomorphism
\[\mathrm{Rag}:\cl(\oo G)=\frac{\hom_\Omega(R_G,\mathbb J(K^c))}{\hom_\Omega(R_G,{K^c}^\times)\Det(\mathbb{U}(\oo G))}\to\frac{\h(\mathbb{A}(K)G)}{\h(KG)\h(\mathbb{A}(\oo)G)}.\]
\end{prop}
\begin{proof}
We need only observe from (\ref{diag8}) that the image under $\mathrm{rag}$ of $\prod_v\Det((\oo_vG)^\times)$ is contained in $\prod_v\Det(((\oo_v^cG)^\times/G)^{\Omega_v})$.
\end{proof}
\begin{defn}
Let $\mathrm{Rag'}$ be the composite:
\[\mathrm{Rag'}:\cl(\oo G)\to \frac{\h(\mathbb{A}(K)G)}{\h(KG)\h(\mathbb{A}(\oo)G)\Theta_{\oG}^t(\mathbb J(K\oL))}.\]
\end{defn}
In 1993, I (still) believe I proved
\begin{equation}\label{eq11}R(\oo G)\subseteq \ker(\mathrm{Rag'}).\end{equation}

\begin{prop}
Let $M$ be a locally free $\oo G$-module and suppose the class $\mathrm{cl}_{\oo G}(M)$ in $\cl(\oo G)$ is represented by $f\in\hom_\Omega(R_G,\mathbb{J}(K^c))$. Then the Steinitz class $\mathrm{cl}_\oo(M)$ in $\cl(\oo)$ is represented by $f(\rho_G)=\prod_{\chi\in\widehat G}f(\chi)^{\chi(1)}$ in $\mathbb{J}(K)$. (Here, $\widehat G$ is the set of absolutely irreducible characters of $G$.)
\end{prop}
\begin{proof}
The group ring $\oo[1]$ of the trivial subgroup $1<G$ is $\oo$ and the inclusion map $\oo\hookrightarrow\oo G$ induces the restriction map $\cl(\oo G)\to\cl(\oo)$, where $\mathrm{cl}_{\oo G}(M)\mapsto\mathrm{cl}_\oo(M)$. By \cite[Theorem 12, p. 63]{Froehlich}, if $f$ represents $\mathrm{cl}_{\oo G}(M)$, then $\res^G_1f\in\hom_\Omega(R_1,\mathbb J(K^c))=\hom(R_1,\mathbb J(K))$ (which can be identified with $\mathbb{J}(K)$ by $g\leftrightarrow g(\chi_0)$, where $\chi_0$ is the trivial character of the trivial group $1$.) But by th formula for $\res$ on p. 62 of \cite{Froehlich}, we have
\[(\res^G_1f)(\chi_0)=f(\ind^G_1\chi_0)=f(\rho_G).\]
But, finally, $\rho_G=\sum_{\chi\in\widehat G}\chi(1)\chi$, so $f(\rho_G)=\prod_{\chi\in\widehat G}f(\chi)^{\chi(1)}$, as required.
\end{proof}
Now, suppose $f\in\hom_\Omega(R_G,\mathbb{J}(K^c))$ represents a class in $R(\oo G)$. Then, by (\ref{eq11}), we have
\[\begin{split}& \mathrm{rag}f=g\Theta_{\oG}^t(h),\text{ where}\\
&g\in \h(KG)\h(\mathbb{A}(\oo)G)\\
&h\in\mathbb{J}(K\oL).
\end{split}\]
But
\[\mathrm{rag}f(\rho_G)=f(\rho_G)\text{ if }\rho_G\in A_G\]
and, in any case,
\[\mathrm{rag}f(2\rho_G)=f(2\rho_G)=f(\rho_G)^2.\]
Likewise, by Proposition \ref{prop9},
\[g\in\hom_\Omega(A_G,{K^c}^\times)\hom_\Omega(A_G,\mathbb{U}(\oo^c)),\]
so, since $\rho_G$ is $\Omega$-stable,
\[g(\rho_G)\in K^\times\mathbb{U}(\oo)\text{ if }\rho_G\in A_G\]
or
\[g(\rho_G)^2\in K^\times\mathbb{U}(\oo).\]
Also
\[\Theta_{\oG}^th(\rho_G)=h(\Theta_{\oG}(\rho_G))\]
or
\[\Theta_{\oG}^th(2\rho_G)=h(\Theta_{\oG}(2\rho_G)).\]
Since $\cl(\oo)=\frac{\mathbb{J}(K)}{K^\times\mathbb{U}(\oo)}$, it follows that $f(\rho_G)$ and $h(\Theta_{\oG}^t(\rho_G))$ (or $f(2\rho_G)$ and $h(\Theta_{\oG}^t(2\rho_G))$) represent the same class in $\cl(\oo)$.

Now, recall $h\in\mathbb{J}(K\oL)=\hom_\Omega(\Z\oG_{\kappa^{-1}},\mathbb{J}(K^c))$ where $K\oL=\map_\Omega(\oG_{\kappa^{-1}},{K^c})$.

Let $S$ be a set of elements $s$ of $G$ whose conjugacy classes $\{\bar s|s\in G\}$ form a set of representatives of the $\Omega$-orbits of $\oG_{\kappa^{-1}}$. For each $s\in S$, let $\Omega_{\bar s}$ be the $\Omega$-stabilizer of $\bar s$, and $K(\bar s)=(K^c)^{\Omega_{\bar s}}$. Then
\[K\Lambda\cong\prod_{s\in S}K(\bar s)\]
where, for each $s\in S$, the projection
\[K\Lambda=\map_\Omega(\oG_{\kappa^{-1}},K^c)\to K(\bar s)\]
is given by evaluation at $\bar s$. Now, if $s$ has order $m$, then $\Omega$ acts on $s$ via $\kappa^{-1}:\Omega\to(\Z/m\Z)^\times$, and the $\Omega$-stabilizer $\Omega_s$ of $s$ is a subgroup of $\Omega_{\bar s}$, so $K(\bar s)$ is a subfield of $(K^c)^{\Omega_s}=K(\zeta_m)$, where $\zeta_m$ is an $m$-th root of $1$.
\begin{defn}
Let $\mu:N_G(\langle s \rangle)\to(\Z/m\Z)^\times$ ($=\aut\langle s\rangle$) give the action of $N_G(\langle s \rangle)$ by conjugation on $\langle s\rangle$.
\end{defn}
Then clearly $\langle s\rangle\cap\bar s=\{s^r|r\in\mu(N_G(\langle s\rangle))\}$, so $\omega\in\Omega$ stabilizes $\bar s$ if and only if $\kappa^{-1}(\omega)=r$ for some $r\in\mu(N_G(\langle s\rangle))$. Hence we have
\begin{equation}\label{eq13}\Omega_{\bar s}/\Omega_s\cong\gal(K(\zeta_m)/K(\bar s))\cong \kappa(\Omega)\cap\mu(N_G(\langle s\rangle))\end{equation}
and since $\Omega/\Omega_s\cong\gal(K(\zeta_m)/K)\cong \kappa(\Omega)$ we have
\begin{equation}\label{eq14}\Omega/\Omega_{\bar s}\cong \gal(K(\bar s)/K)\cong\kappa(\Omega)/\kappa(\Omega)\cap\mu(N_G(\langle s\rangle)).\end{equation}
Also $\Omega/\Omega_{\bar s}$ is isomorphic as an $\Omega$-set to the $\Omega$-orbit $\Omega\cdot\bar s$ of $\bar s$. Recall that by Corollary \ref{coroll2.1},
\[\Theta_{\oG}(\rho_G)=\sum_{m\mid e}\frac{n}{m}\frac{m-1}{2}\sum_{\substack{\bar s\in \oG_{\kappa^{-1}}\\ |s|=m}}\bar s.\]
\begin{prop}
For $h\in\mathbb{J}(K\oL)$, if $\rho_G\in A_G$, then
\[h(\Theta_{\oG}(\rho_G))=\prod_{m|e}\left(\prod_{\substack{s\in S\\|s|=m}}N_{K(\bar s)/K}(h(\bar s))\right)^{\frac{n}{m}\frac{m-1}{2}}.\]
In any case,
\[h(\Theta_{\oG}(2\rho_G))=\prod_{m|e}\left(\prod_{\substack{s\in S\\|s|=m}}N_{K(\bar s)/K}(h(\bar s))\right)^{\frac{n}{m}(m-1)}.\]
\end{prop}
\begin{proof}
We shall prove this for $\rho_G\in A_G$. (The general case will then also be clear.) Then
\[h(\Theta_{\oG}(\rho_G))=\prod_{m|e}\left(\prod_{\substack{\bar s\in \oG_{\kappa^{-1}}\\|s|=m}}h(\bar s)\right)^{\frac{n}{m}\frac{m-1}{2}}.\]
We organize the inner product into $G$-orbits and put
\[\prod_{\substack{\bar t\in\oG_{\kappa^{-1}}\\|t|=m}}h(\bar t)=\prod_{\substack{s\in S\\|s|=m}}\prod_{\bar t\in\Omega\cdot\bar s}h(\bar t).\]
Now $\Omega\cdot \bar s=\{\bar s^\omega|\omega\in\Omega/\Omega_{\bar s}\}$, so the inner product is
\[\prod_{\bar t\in\Omega\cdot \bar s}h(\bar t)=\prod_{\omega\in\Omega/\Omega_{\bar s}}h(\bar s^\omega)=\prod_{\omega\in\gal(K(\bar s)/K)}h(\bar s)^\omega=N_{K(\bar s)/K}h(\bar s).\]
(Note that $\bar s^\omega=\overline{s^{\kappa^1{\omega}}}$.)
\end{proof}
\begin{prop}\label{prop16}
If $\rho_G\in A_G$, then
\[\rt(\oo,G)\subseteq\left(\prod_{m|e}\prod_{\substack{s\in S\\|s|=m}}N_{K(\bar s)/K}(\cl(\oo_{K(\bar s)}))^{\frac{e}{m}\frac{m-1}{2}}\right)^{\frac{n}{e}}.\]
In any case
\[\rt(\oo,G)^2\subseteq\left(\prod_{m|e}\prod_{\substack{s\in S\\|s|=m}}N_{K(\bar s)/K}(\cl(\oo_{K(\bar s)}))^{\frac{e}{m}(m-1)}\right)^{\frac{n}{e}}.\]
\end{prop}

\begin{coroll}\label{coroll1}
If $G$ is abelian of odd order or with noncyclic Sylow $2$-subgroup, then
\[\rt(\oo,G)\subseteq\left(\prod_{m|e}N_{K(\zeta_m)/K}(\cl(\oo_{K(\zeta_m)}))^{\frac{e}{m}\frac{m-1}{2}}\right)^{\frac{n}{e}}.\]
Otherwise,
\[\rt(\oo,G)^2\subseteq\left(\prod_{m|e}N_{K(\zeta_m)/K}(\cl(\oo_{K(\zeta_m)}))^{\frac{e}{m}(m-1)}\right)^{\frac{n}{e}}.\]
\end{coroll}
\begin{proof}
If $G$ is abelian, then for all $s\in G$, $\{s\}=\bar s$. Then $\Omega_{\bar s}=\Omega_s$, so $K(\bar s)=K(\zeta_m)$, where $|s|=m$, by (\ref{eq13}). Furthermore if $m|e$, there is an element in $G$ of order $m$. The result follows.
\end{proof}
Note: If $G$ is abelian of odd order, Endo shows equality in the corollary (\cite[Chapter II, 1.4, p. 31]{Endo}).
\begin{defn}
Let \[c(e)=\gcd_{\substack{p|e\\p\text{ prime}}}(p-1).\]
\end{defn}

\begin{coroll}\label{coroll2}
Suppose $G$ is abelian and $K$ contains the $e$-th roots of $1$. If $|G|$ is odd or $S_2(G)$ is noncyclic then
\[\rt(\oo,G)\subseteq\cl(\oo)^{\frac{n}{e}\frac{c(e)}{2}}.\]
If $S_2(G)$ is nontrivial cyclic, then
\[\rt(\oo,G)\subseteq\cl(\oo)^{\frac{n}{e}c(e)}\ (=\cl(\oo))^\frac{n}{e}.\]
\end{coroll}
\begin{proof}
If $|G|$ is odd and $S_2(G)$ is noncyclic, then by Corollary \ref{coroll1},
\[\rt(\oo,G)\subseteq\cl(\oo)^{\frac{n}{2e}\gcd_{m|e}\frac{e}{m}(m-1)}.\]
If $S_2(G)$ is cyclic nontrivial, then
\[\rt(\oo,G)^2\subseteq\cl(\oo)^{\frac{n}{e}\gcd_{m|e}\frac{e}{m}(m-1)}.\]
Claim: $\gcd_{m|e}\frac{e}{m}(m-1)=c(e)$.
First note that if $p^r||e$, then $p\nmid \frac{e}{p^r}(p^r-1)$, so no prime factor of $e$ divides the given "gcd". Hence $\gcd_{m|e}\frac{e}{m}(m-1)=\gcd_{m|e}(m-1)$. Clearly $\gcd_{m|e}(m-1)|c(e)$. Also, clearly if $m|e$, then $c(e)|c(m)$. It suffices to show that $c(m)|(m-1)$. We show this by induction. (Note that $c(1)=0$, also if $p$ is prime, $c(p)=p-1$.) Suppose the assertion true for all numbers less than $m$ ($>1$). If $p|m$, then $m-1=p\left(\frac{m}{p}-1\right)+p-1$, and by induction, $c(m)|c\left(\frac{m}{p}\right)|\left(\frac{m}{p}-1\right)$. Clearly $c(m)|p-1$, so $c(m)|m-1$. Hence the first assertion holds. Similarly it follows that if $S_2(G)$ is cyclic nontrivial, then
\[\rt(\oo,G)^2\subseteq\cl(\oo)^{\frac{n}{e}\gcd_{m|e}\frac{e}{m}(m-1)}=\cl(\oo)^{\frac{n}{e}c(e)}=\cl(\oo)^\frac{n}{e}.\]
But since $S_2(G)$ is cyclic, $2\nmid\frac{n}{e}$, so $\frac{n}{e}$ is odd. Hence $[\cl(\oo):\cl(\oo)^\frac{n}{e}]$ is odd. Then, clearly,
\[\rt(\oo,G)\subseteq\cl(\oo)^\frac{n}{e}.\]
\end{proof}

\begin{remark}
To compare the results with Long's note that he defines
\[d(e)=\begin{cases}\gcd_{p|e}\frac{p-1}{2}&\text{if $e$ is odd}\\1&\text{if $e$ is even}.\end{cases}\]
Hence
\[c(e)=\begin{cases}2d(e)&\text{if $e$ is odd}\\d(e)=1&\text{if $e$ is even}.\end{cases}\]
Using this notation our Corollary \ref{coroll2} becomes
\[\rt(\oo,G)\subseteq\begin{cases}\cl(\oo)^{\frac{n}{e}d(e)}&\text{if $|G|$ is odd}\\
\cl(\oo)^\frac{n}{2e}&\text{if $S_2(G)$ is noncyclic}\\
\cl(\oo)^\frac{n}{e}&\text{if $S_2(G)$ is cyclic nontrivial}
.\end{cases}\]
Long \cite[Chapter I, Thm 3, p. 15]{LongPhD} obtains the following
\[\rt(\oo,G)=\begin{cases}\cl(\oo)^{\frac{n}{e}d(e)}&\text{if $|G|$ is odd}\\
\cl(\oo)^\frac{n}{2e}&\text{if $S_2(G)$ is noncyclic and the highest elementary divisor}\\ 
&\text{occurs at least two times in $S_2(G)$}\\
\cl(\oo)^\frac{n}{e}&\text{if $S_2(G)$ is noncyclic and the highest elementary divisor} \\
&\text{occurs only once}\\
\cl(\oo)^\frac{n}{e}&\text{if $S_2(G)$ is nontrivial cyclic}
.\end{cases}\]
Thus the upper bound of Corollary \ref{coroll2} is actually attained except in the case where $S_2(G)$ is noncyclic, but the highest elementary divisor occurs only once. This shows that, in general, the use of (\ref{eq11}), even in the abelian case (where equality holds) will not immediately give an exact answer for $\rt(\oo,G)$. The difficulty, apparently, is that if $h\in\mathbb J(K\oL)$, then the coset $\h(KG)\h(\mathbb{A}(\oo)G)\Theta_{\oG}^t(h)$ in $\h(\mathbb{A}(KG))$ may be disjoint from the image of $\hom_\Omega(R_G,\mathbb J(K^c))$ under the map
\[\mathrm{rag}:\hom_\Omega(R_G,\mathbb{J}(K^c))\to\hom_\Omega(A_G,\mathbb J(K^c)).\]
This leads me to believe the problem has to do with Grunwald-Wang phenomena.
\end{remark}
One situation in which the upper bound of Proposition \ref{prop16} is exact is for metacyclic groups. In particular if
\[G=\langle s,t|s^p=t^q=1,tst^{-1}=s^r\rangle\text{ ($p$ and $q$ odd primes),}\]
where $r\in(\Z/p\Z)^\times$ has order $q$, we have 
\begin{coroll}
Suppose $K$ and $\Q(\zeta_{pq})$ are linearly disjoint over $\Q$. Let $F=K(\zeta_q)$ and $L/K$ be the subextension of $K(\zeta_p)/K$ of codegree $q$. Then
\[\rt(\oo,G)\subseteq N_{F/K}(\cl(\oo_F))^{p\frac{(q-1)}{2}}N_{L/K}(\cl(\oo_L))^{q\frac{p-1}{2}}.\]
\end{coroll}
\begin{proof}
Consider $\bar t$. Since $\langle t\rangle=N_G(\langle t\rangle)$ it follows that $\mu:N_G(\langle t\rangle)\to(\Z/q\Z)^\times$ is trivial. Also, since $\Q(\zeta_q)$ and $K$ are linearly disjoint, $\kappa(\Omega)=(\Z/q\Z)^\times$, so by (\ref{eq14})
\[\Omega/\Omega_{\bar t}\cong\gal(K(\bar t)/K)\cong(\Z/q\Z)^\times,\]
so $K(\bar t)=K(\zeta_q)=F$. Since $n=e=pq$ and $m=|t|=q$, it follows that the corresponding factor in the statement of Proposition \ref{prop16} is
\[N_{K(\bar t)/K}(\cl(\oo_{K(\bar t)}))^{\frac{e}{m}\frac{m-1}{2}}=N_{F/K}(\cl(\oo_F))^{p\frac{q-1}{2}}.\]
The same factor clearly arises for all elements of order $q$.

Now consider $\bar s$. $N_G(\langle s\rangle)=G$, and clearly $\mu(G)$ in $(\Z/p\Z)^\times$ is generated by $r$ and so has order $q$. Again, linear disjointness implies $\kappa(\Omega)=(\Z/p\Z)^\times$, so (\ref{eq14}) gives
\[\Omega/\Omega_{\bar s}\cong\gal(K(\bar s)/K)\cong(\Z/p\Z)^\times/\langle r\rangle.\]
Then $\gal(K(\zeta_p)/K(\bar s))\cong \langle r\rangle$, so $K(\bar s)=L$, and the corresponding factor for $\bar s$ is
\[N_{L/K}(\cl(\oo_L))^{q\frac{p-1}{2}}.\]
Again, the same holds for other elements of order $p$, and since there are no elements of order $pq$, the Corollary is proved.
\end{proof}
\begin{remark}
The same approach with a bit more effort gives Endo's results (\cite{Endo}, of course again Endo's results are exact and Proposition \ref{prop16} only gives the upper bound) for metacyclic groups of order $qp^a$, where $q$ is not necessarily prime, but $p$ and $q$ are odd. (Here $s$ has order $p^a$, $t$ has order $q|p-1$, not necessarily prime, but odd, and $tst^{-1}=s^r$, where $r$ in $(\Z/p^a\Z)^\times$ has order $q$.)
\end{remark}

\appendix
\begin{section}{Appendix: comparison of $K(\bar s)$ and $E_{K,G,s}$}
In this appendix, we show that McCulloh's $K(\bar s)$ coincides with the extension $E_{K,G,s}$ defined in \cite{CaputoCobbe1}. As a consequence, Proposition \ref{prop16} of this report will be shown to be a slightly weaker form of \cite[Theorem 2.10]{CaputoCobbe1}.\\
By definition $K(\bar s)=(K^c)^{\Omega_{\bar s}}$, where
\[\Omega_{\bar s}=\mathrm{Stab}_\Omega \bar s=\{\omega\in\Omega:\ \omega(\bar s)=\bar s\}=\{\omega\in\Omega\ |\ \exists t\in G\ :\ \omega(s)=tst^{-1}\}.\]
We recall that the action is defined (on $s$ and hence on $\bar s$) by $\omega(s)=s^{\kappa^{-1}(\omega)}=s^{\kappa(\omega^{-1})}$, where $\kappa:\Omega\to(\Z/e\Z)^\times$ is the cyclotomic character and $e$ is the exponent of $G$. Then
\[\begin{split}\Omega_{\bar s}&\supseteq \{\omega\in\Omega:\ \omega(s)=s\}=\mathrm{Stab}_{\Omega}s=\{\omega\in\Omega:\ \kappa(\omega^{-1})\equiv1\pmod{o(s)}\}\\&=\{\omega\in\Omega:\ \omega(\zeta_{o(s)})=\zeta_{o(s)}\}.\end{split}\]
Therefore
\[K(\bar s)\subseteq (K^c)^{\mathrm{Stab}_{\Omega}s}=K(\zeta_{o(s)}).\]
Further
\[\begin{split}\gal(K(\zeta_{o(s)})/K(\bar s))&=\gal(K^c/K(\bar s))/\gal(K^c/K(\zeta_{o(s)}))=\mathrm{Stab}_{\Omega}\bar s/\mathrm{Stab}_{\Omega}s\\&=\mathrm{Stab}_{\Omega}\bar s|_{K(\zeta_{o(s)})}.\end{split}\]
On the other hand, by definition, $E_{K,G,s}/K$ is the subextension of $K(\zeta_{o(s)})/K$ fixed by the subgroup $\nu_{K,s}^{-1}(\mu_s(N_G(s)))$ of $\gal(K(\zeta_{o(s)})/K)$.
Here $N_G(s)$ is the normalizer of $s$ in $G$, $\mu:N_G(s)\to(\Z/o(s)\Z)^\times$ is defined by $tst^{-1}=s^{\mu_s(t)}$ for $t\in N_G(s)$ and $\nu_{K,s}:\gal(K(\zeta_{o(s)})/K)\to(\Z/o(s)\Z)^\times$ is the cyclotomic character (note that here $-1$ denotes a counterimage and not a multiplicative inverse as for $\kappa^{-1}$). In particular
\[\gal(K(\zeta_{o(s)})/E_{K,G,s})=\nu_{K,s}^{-1}(\mu_s(N_G(s))).\]
Now
\[\begin{split}\mathrm{Stab}_\Omega\bar s|_{K(\zeta_{o(s)})}&=\{\omega\in\gal(K(\zeta_{o(s)})/K)\ |\ \exists\, t\in N_G(s)\ :\ \omega(s)=tst^{-1}=s^{\mu_s(t)}\}\\
&=\{\omega\in\gal(K(\zeta_{o(s)})/K)\ |\ \exists\, t\in N_G(s)\ :\ \nu_{K,s}(\omega^{-1})=\mu_s(t)\}\\
&=\{\omega\in\gal(K(\zeta_{o(s)})/K)\ |\ \exists\, t\in N_G(s)\ :\ \nu_{K,s}(\omega)=\mu_s(t^{-1})\}\\
&=\{\omega\in\gal(K(\zeta_{o(s)})/K)\ |\ \exists\, t\in N_G(s)\ :\ \nu_{K,s}(\omega)=\mu_s(t)\}\\&
=\nu_{K,s}^{-1}(\mu_s(N_G(s))).
\end{split}\]
Hence
\[K(\bar s)=E_{K,G,s}.\]
It is now clear that, in the case of groups of odd order or with non-cyclic $2$-Sylow subgroup, Proposition \ref{prop16} of this report is equivalent to the inclusion of \cite[Theorem 2.10]{CaputoCobbe1}, namely
\[\rt(K,G)\subseteq\W(K,G),\]
where 
\[\W(K,G)=\prod_{\tau\in G^*}W(k,E_{k,G,\tau})^{\frac{o(\tau)-1}{2}\frac{\#G}{o(\tau)}}=\prod_{\ell|\#G}\ \prod_{\sigma\in G\{\ell\}^*}W(k,E_{k,G,\sigma})^{\frac{\ell-1}{2}\frac{\#G}{o(\sigma)}}\subseteq\cl(k)\]
and
\[W(k,E_{k,G,\tau})=N_{E_{k,G,\tau}/k}\cl(E_{k,G,\tau}).\]

\end{section}

\bibliography{bibMcCulloh}
\addcontentsline{toc}{section}{Bibliography}
\bibliographystyle{abbrv}

Transcription and Appendix A by L. Caputo and A. Cobbe

\end{document}